\documentclass[preprint,12pt]{elsarticle}

\usepackage[utf8]{inputenc}
\usepackage{amsmath}
\usepackage{amscd}
\usepackage{amsthm}
\usepackage{ upgreek }
\usepackage{amssymb}
\usepackage{hyperref}
\usepackage{enumitem}
\usepackage{url}
\usepackage{times}

\newtheorem{lemma}{Lemma}[section]

\numberwithin{equation}{section}

\journal{Journal of Number Theory}
\begin{document}
	\begin{frontmatter}
	
	

	\title{On the discrete mean square of certain hybrid sum involving $a_{\mathbb{K}}(n)$}
	
	
	\author[label1]{EKTA SONI}
	\author[label2]{M.S. DATT}
	\author[label3]{AYYADURAI SANKARANARAYANAN}
	\affiliation[label1]{organization={University of Hyderabad},
		addressline={CR Rao Road, Gachibowli}, 
		city={Hyderabad},
		postcode={500046}, 
		state={Telangana},
		country={India}}
	\affiliation[label2]{organization={University of Hyderabad},
		addressline={CR Rao Road, Gachibowli}, 
		city={Hyderabad},
		postcode={500046}, 
		state={Telangana},
		country={India}}
		\affiliation[label3]{organization={University of Hyderabad},
			addressline={CR Rao Road, Gachibowli}, 
			city={Hyderabad},
			postcode={500046}, 
			state={Telangana},
			country={India}}
	\begin{abstract}
	Let $\mathbb{K}$ be a non-normal algebraic number field of cubic degree given by the polynomial $x^{3}+ax^{2}+bx+c$ of discriminant $D_{\mathbb{K}}<0$. For sufficiently large $x$, we establish an asymptotic formula for the hybrid sum
$$\sum\limits_{\substack{n= \sum_{i=1}^{8}n_{i}^{2}\leq x\\ (n_{1},n_{2},n_{3},n_{4},n_{5},n_{6},n_{7},n_{8})\in \mathbb{Z}^{8}  }} a_{\mathbb{K}}^{2}(n)$$
with a tight error term.
\end{abstract}




\end{frontmatter}

	\section{Introduction}
Let $\mathbb{K}$ be a number field of degree $m$ over $\mathbb{Q}$. Let $\mathcal{O}_{\mathbb{K}}$ be the ring of integers in $\mathbb{K}$. The Dedekind zeta function associated to $\mathbb{K}$ is defined as 
\begin{center}
	$\zeta_{\mathbb{K}}(s)=\prod\limits_{\substack{\mathcal{P}\subseteq \mathcal{O}_{\mathbb{K}}\\ \mathcal{P} \neq (0)}}\big(1-\frac{1}{N(\mathcal{P})}\big)^{-1}=\sum\limits_{\substack{\mathfrak{a}\subseteq \mathcal{O}_{\mathbb{K}}\\\mathfrak{a}\neq (0)}}\frac{1}{N^{s}(\mathfrak{a})}$,
\end{center}
for any $s\in \mathbb{C}$, $Re(s)>1$, where the product is taken over all non-zero prime ideals in $\mathcal{O}_{\mathbb{K}}$ and $N(\mathfrak{a})$ is the absolute norm of non-zero integral ideals.\\
The Dedekind zeta function can be extended meromorphically to the whole complex plane with a simple pole at $Re(s)=1$. The residue of the Dedekind zeta function at $s=1$ is given by
\begin{center}
	$\frac{2^{r_{1}}(2\pi)^{r_{2}}hR}{\omega \sqrt{|D_{\mathbb{K}}|}}$,
\end{center} 
where $r_{1}$ is the total number of real embeddings and $r_{2}$ is the total number of complex embeddings, $h$ is the class number, $R$ denotes the regulator, $\omega$ is the number of units in $\mathcal{O}_{\mathbb{K}}$ and $|D_{\mathbb{K}}|$ is the absolute value of the discriminant of $\mathbb{K}$.\\

For $n\in \mathbb{N}$, let $a_{\mathbb{K}}(n)$ be the number of integral ideals of norm $n$ in $\mathcal{O}_{\mathbb{K}}$, then Dedekind zeta function $\zeta_{\mathbb{K}}(s)$ can be written as
\begin{center}
	$\zeta_{\mathbb{K}}(s)=\sum\limits_{n=1}^{\infty}\frac{a_{\mathbb{K}}(n)}{n^{s}},~~~~~Re(s)>1.$
\end{center}
As shown by Chandrasekharan and Good \cite{CKAG}, the coefficients $a_{\mathbb{K}}(n)$ are multiplicative and satisfy the following upper bound
\begin{center}
	$a_{\mathbb{K}}(n)\leq d(n)^{[\mathbb{K}:\mathbb{Q}]}\ll n^{\epsilon}$,
\end{center}
where $d(n)$ is the divisor function. As $a_{\mathbb{K}}(n)$ is multiplicative, we have
\begin{center}
	$\displaystyle \zeta_{\mathbb{K}}(s)=\prod\limits_{p}\Big(1+\frac{a_{\mathbb{K}}(p)}{p^{s}}+\frac{a_{\mathbb{K}}(p^{2})}{p^{2s}}+\dots+\frac{a_{\mathbb{K}}(p^{\ell})}{p^{\ell s}}+\dots\Big)$,
\end{center}
for $Re(s)>1$.\\
In $1949$, Landau\cite{Land}, in $1963$ Chandrasekharan and Narasimhan \cite{CKN} , and in $1983$ chandrasekharan and Good \cite{CKAG} considered first, second and $\ell^{th}$-moments $(\ell \geq 2)$ of $a_{\mathbb{K}}(n)$'s respectively and obtained asymptotic formulae.\\

For a cubic non-normal extension $\mathbb{K}$ over $\mathbb{Q}$, given by an irreducible polynomial $x^{3}+ax^{2}+bx+c\in \mathbb{Z}[x]$ such that the discriminant $D_{\mathbb{K}}<0$,
recently Naveen and Prashant \cite{P tiwari} studied the Problem for $\ell^{th}$-order moments of coefficients of the Dedekind zeta function over certain sequence of natural numbers and obtained asymptotic formulae, for a cubic non-normal extension over $\mathbb{Q}$.\\  

In this paper, we study the second order moments of coefficients of Dedekind zeta function over certain sequence of natural numbers. More precisely, we would like to investigate the asymptotic behaviour of $\sum\limits_{\substack{ n= \sum_{i=1}^{8}n_{i}^{2}\leq x\\ (n_{1},\dots,n_{8})\in \mathbb{Z}^{8}  }} a_{\mathbb{K}}^{2}(n)$ for sufficiently large $x$ and establish an asymptotic formula with a tight error term.\\

\noindent Indeed we prove the following:\\
\noindent\textbf{Main Theorem}:	\textit{ For $x\geq x_{0}$, where $x_{0}$ is sufficiently large, and $\epsilon>0$ any small constant, we have
	\begin{align*} 
		\sum\limits_{\substack{n= \sum_{i=1}^{8}n_{i}^{2}\leq x\\ (n_{1},n_{2},n_{3},n_{4},n_{5},n_{6},n_{7},n_{8})\in \mathbb{Z}^{8}  }} a_{\mathbb{K}}^{2}(n)&=C x^{4} P_{1}(\log{x})+O(x^{\frac{198}{53}+\epsilon}),
	\end{align*}
	where $P_{1}(\log{x})$ is a polynomial of degree $1$ in $\log{x}$.
}\\

The paper is planned as follows. In section $2$, we introduce some preliminaries and important lemmas that will play crucial role in the proof of our result. In section $3$, the proof of the theorem is presented. Throughout the paper, $\epsilon$ denotes sufficiently small positive constant, not necessarily the same at each occurrence and the implied constants may depend on $\epsilon,\,f$ and $|D_{\mathbb{K}}|$ and effective.  
\section{Preliminaries and Some Lemmas}
Let $k$ be a positive even integer and $H_{k}(SL(2,\mathbb{Z}))$ be the set of all normalized Hecke eigenforms of weight $k$ for the full modular group $SL(2,\mathbb{Z})$. For a cusp form $f\in H_{k}(SL(2,\mathbb{Z}))$, let $\lambda_{f}(n)$ be the normalized $n^{th}$ coefficients of Fourier expansion of $f$ at infinity, i.e, 
\begin{align*}
	f(z)&=\sum_{n=1}^{\infty} \lambda_{f}(n)n^{\frac{k-1}{2}}e^{2\pi i nz},~~~~\text{for all}~z\in{\mathbb{H}},
\end{align*}
where $\mathbb{H}$ be the Poincar\'e upper Half-plane. These Fourier coefficients $\lambda_{f}(n)$ are the Hecke eigenvalues of $f$ and $\lambda_{f}(n)\in \mathbb{R}$ $\forall~~n.$ The $\lambda_{f}(n)$ satisfies the following relation. For $m,n\in \mathbb{N}$
\begin{align*}
	\lambda_{f}(m)\lambda_{f}(n)&=\sum\limits_{d|gcd(m,n)}\lambda_{f}\Big(\frac{mn}{d^{2}}\Big).
\end{align*}
Thus $\lambda_{f}(n)$ is a multiplicative function.\\
In 1974, P. Deligne \cite{PD} proved that for any prime $p$, there exist complex numbers $\alpha_{f}(p),\,\beta_{f}(p)$ such that
\begin{align*}
	\lambda_{f}(p)&=\alpha_{f}(p)+\beta_{f}(p),\\
	|\alpha_{f}(p)|&=|\beta_{f}(p)|=\alpha_{f}(p)\beta_{f}(p)=1.
\end{align*}
Given any Dirichlet's character $\chi$ of Modulo $N$, we have the Dirichlet's $L$-function, namely
\begin{center}
	$L(s,\chi)=\sum\limits_{n=1}^{\infty}\frac{\chi(n)}{n^{s}}~~~\text{for}~Re(s)>1$,
\end{center}
and
\begin{align*}
	L(s,sym^{j}f)&=\prod_{p}\Big(\prod\limits_{i=0}^{j}\Big(1-\frac{\alpha^{j-i}_{f}(p)\beta^{i}_{f}(p)}{p^{s}}\Big)\Big)^{-1}\\
	&=\sum_{n=1}^{\infty}\frac{\lambda_{sym^{j}f}(n)}{n^{s}}~~~\text{for}~Re(s)>1.
\end{align*}
Note that $\lambda_{sym^{j}f}(n)$ is also real valued multiplicative function. We also have the Deligne's bound $\lambda_{f}(n)\leq d(n)\leq n^{\epsilon}$, $d(n)$ is the divisor function.\\
We note that the $L$-function associated to $\lambda_{f}(n)$ is defined as
\begin{align*}
	L(s,f)&=\sum_{n=1}^{\infty}\frac{\lambda_{f}(n)}{n^{s}}, 
\end{align*}
for $Re(s)>1,$ where $\lambda_{f}(n)$ is the eigenvalue of Hecke operators $T_{n}$ and the symmetric square $L$-function is defined as
\begin{align*}
	L(s,sym^{2}f)&=\sum\limits_{n=1}^{\infty}\frac{\lambda_{sym^{2}f}(n)}{n^{s}}\\
	&=\prod_{p} \Big(1-\frac{\alpha^{2}(p)}{p^{s}}\Big)^{-1}\Big(1-\frac{\beta^{2}(p)}{p^{s}}\Big)^{-1}\Big(1-\frac{1}{p^{s}}\Big)^{-1},
\end{align*}
for $Re(s)>1,$ where $\lambda_{sym^{2}f}(n)$ is multiplicative.\\
Let $N\in \mathbb{N}$ and $\chi$ be a Dirichlet character $(\mod N)$. Let $\Gamma_{0}(N)=\Big\{ \begin{pmatrix}
	a & b\\
	c & d
\end{pmatrix}
\in SL(2,\mathbb{Z})~\big|~\text{$c$ is a multiple of $N$}\Big\}$. The subgroup $\Gamma_{0}(N)$ is called the congruence subgroup of $SL(2,\mathbb{Z})$. For $\begin{pmatrix}
	a & b\\
	c & d
\end{pmatrix} \in \Gamma_{0}(N) $, and $z\in \mathbb{H}$ (the upper half plane), if $f \begin{pmatrix}
	az+b\\
	cz+d
\end{pmatrix}=\chi(d)(cz+d)^{k}f(z)$, then $f$ is known as the modular form of weight $k$ of level $N$ with Nebentypus $\chi$. \\
From \cite{P tiwari}, for a holomorphic cusp form  $f$ of weight $1$ of the congruence subgroup $\Gamma_{0}(N)$ with $N=|D_{\mathbb{K}}|$, we have 
\begin{center}
	$\zeta_{\mathbb{K}}(s)=\zeta(s)L(s,f),$
\end{center}
and hence,
\begin{center}
	$a_{\mathbb{K}}(n)=\sum\limits_{d|n}\lambda_{f}(d).$
\end{center}
This implies that $a_{\mathbb{K}}(p)=1+\lambda_{f}(p).$\\

\noindent Let
\begin{center}
	$	q(x_{1},x_{2},x_{3},x_{4},x_{5},x_{6},x_{7},x_{8})=\sum\limits_{i=1}^{8}x_{i}^{2}\in \mathbb{Z}[x_{1},x_{2},x_{3},x_{4},x_{5},x_{6},x_{7},x_{8}],$
\end{center} 
and
\begin{center}
	$r_{8}(n)=\#\{(a_{1},a_{2},a_{3},a_{4},a_{5},a_{6},a_{7},a_{8})\in \mathbb{Z}^{8}~|~ \sum_{i=1}^{8} a_{i}^{2}=n\}. $
\end{center}
From \cite{Jacobi}, we know that $r_{8}(n)=16(-1)^{n}\sum_{d|n}(-1)^{d}d^{3}$. First, we show that $(-1)^{n}\sum_{d|n}(-1)^{d}d^{3}$ is a multiplicative function.
\begin{lemma}
	Let $g(n)=(-1)^{n}\sum_{d|n}(-1)^{d}d^{3}$. Then $g(n)$ is a multiplicative arithmetic function.
\end{lemma}
\begin{proof}
	To prove the lemma, we show that for any two distinct primes $p,\,q$ and for $\ell,m \in \mathbb{N}$
	\begin{center}
		$g(p^{\ell}q^{m})=g(p^{\ell})g(q^{m})$.
	\end{center}
	\textbf{Case(i)} In this case we take $p=2$ and $q$ is odd. By the definition of $g(n)$,
	\begin{align*}
		g(2^{\ell})&=(-1)^{2^{\ell}}\{-1+2^{3}+2^{6}+\dots+2^{3\ell}\}\\
		&=\{-1+2^{3}+2^{6}+\dots+2^{3\ell}\},
	\end{align*}
	and 
	\begin{align*}
		g(q^{m})&=(-1)^{q^{m}}\{-1-q^{3}-\dots-q^{3m}\}\\
		&=\{1+q^{3}+\dots+q^{3m}\}.
	\end{align*}
	Therefore,
	\begin{align*}
		g(2^{\ell})g(q^{m})&=\{-1+2^{3}+2^{6}+\dots+2^{3\ell}\}\{1+q^{3}+\dots+q^{3m}\}\\
		&=-\{1+q^{3}+\dots+q^{3m}\}+\sum\limits_{\substack{1\leq i\leq \ell \\ 0\leq j \leq m}} 2^{3i} q^{3j}\\
		&=-\sum\limits_{j=0}^{m}q^{3j}+\sum\limits_{\substack{1\leq i\leq \ell \\ 0\leq j \leq m}} 2^{3i} q^{3j}.
	\end{align*}
	Now take,
	\begin{align*}
		g(2^{\ell}q^{m})&=(-1)^{2^{\ell}q^{m}}\big\{-1+\sum\limits_{i=1}^{\ell}2^{3i}-\sum\limits_{j=1}^{m}q^{3j}+\sum\limits_{\substack{1\leq i\leq \ell \\ 1\leq j \leq m}} 2^{3i} q^{3j}\big\}\\
		&=-\big\{1+\sum\limits_{j=1}^{m}q^{3j}\big\}+\big\{\sum\limits_{i=1}^{\ell}2^{3i}+ \sum\limits_{\substack{1\leq i\leq \ell \\ 1\leq j \leq m}} 2^{3i} q^{3j}\big\}\\
		&=-\sum\limits_{j=0}^{m}q^{3j}+\sum\limits_{\substack{1\leq i\leq \ell \\ 0\leq j \leq m}} 2^{3i} q^{3j}.
	\end{align*}
	Therefore, $g(2^{\ell}q^{m})=g(2^{\ell})g(q^{m})$.\vspace{3mm}\\
	\textbf{Case (ii)} Suppose both $p,\,q$ are odd primes.\\
	We have,
	\begin{align*}
		g(p^{\ell})&=(-1)^{p^{\ell}}\{-1-p^{3}-p^{6}\dots-p^{3\ell}\}\\
		&=1+p^{3}+p^{6}+\dots+p^{3\ell},
	\end{align*}
	and 
	\begin{align*}
		g(p^{m})&=1+q^{3}+q^{6}+\dots+q^{3m}.
	\end{align*}
	Therefore,
	\begin{align*}
		g(p^{\ell})g(p^{m})&=\{1+p^{3}+p^{6}+\dots+p^{3\ell}\}\{1+q^{3}+q^{6}+\dots+q^{3m}\}\\
		&=\sum\limits_{\substack{0\leq i\leq \ell \\ 0\leq j \leq m}} p^{3i}q^{3j}.
	\end{align*}
	Now, 
	\begin{align*}
		g(p^{\ell}q^{m})=(-1)^{p^{\ell}q^{m}}\sum\limits_{\substack{0\leq i\leq \ell \\ 0\leq j \leq m}}(-1)^{p^{3i}q^{3j}} p^{3i}q^{3j}=\sum\limits_{\substack{0\leq i\leq \ell \\ 0\leq j \leq m}} p^{3i}q^{3j}.
	\end{align*}
	Therefore,
	$g(p^{\ell}q^{m})=g(p^{\ell})g(q^{m})$.\\
	Hence the Lemma.
\end{proof}

We know that for instance by \cite{OM} , $a_{\mathbb{K}}(n)$ is multiplicative and by the above Lemma $g(n)$ is multiplicative. Hence $a_{\mathbb{K}}^{2}(n)g(n)$ is also multiplicative. To study the asymptotic behaviour of $\sum\limits_{\substack{ n= \sum_{i=1}^{8}n_{i}^{2}\leq x\\ (n_{1},n_{2},n_{3},n_{4},n_{5},n_{6},n_{7},n_{8})\in \mathbb{Z}^{8}  }}a_{\mathbb{K}}^{2}(n)$, we consider the Dirichlet's series associated to $a_{\mathbb{K}}^{2}(n)$ over the sequence of natural numbers given by $r_{8}(n)$. For any $x\in \mathbb{R^{+}},~x\geq 1$, we have
\begin{align*}
	\sum\limits_{\substack{n= \sum_{i=1}^{8}n_{i}^{2}\leq x\\ (n_{1},n_{2},n_{3},n_{4},n_{5},n_{6},n_{7},n_{8})\in \mathbb{Z}^{8}  }}a_{\mathbb{K}}^{2}(n) &=\sum_{n\leq x}a_{\mathbb{K}}^{2}(n)\Big(\sum\limits_{\substack{n= \sum_{i=1}^{8}n_{i}^{2}\\ (n_{1},n_{2},n_{3},n_{4},n_{5},n_{6},n_{7},n_{8})\in \mathbb{Z}^{8}}} 1\Big)\\
	&= \sum_{n\leq x} a_{\mathbb{K}}^{2}(n) r_{8}(n)\\
	&=16\sum_{n\leq x} a_{\mathbb{K}}^{2}(n) g(n).
\end{align*}

\begin{lemma} The $L$-function $F(s)$ associated with $a_{\mathbb{K}}^{2}(n)r_{8}(n)$ $i.e,$
	\begin{align*}
		F(s):=L(s, a_{\mathbb{K}}^{2}(n)r_{8}(n))&=\sum_{n=1}^{\infty}\frac{a_{\mathbb{K}}^{2}(n)r_{8}(n)}{n^{s}},
	\end{align*}
	can be expressed as
	\begin{align*}
		F(s)&=16\frac{B_{2}(s)}{A_{2}(s)}\zeta^{2}(s-3)L^{2}(s-3,f)L(s-3,sym^{2}f)\zeta^{2}(s)L^{2}(s,f)L(s,sym^{2}f)B(s).
	\end{align*}
	Let 
	\begin{center}
		$G(s)=16\zeta^{2}(s)L^{2}(s,f)L(s,sym^{2}f)B(s),$
	\end{center}
	where $B(s)$ is a harmless factor and $G(s)$ converges absolutely for $Re(s)>\frac{7}{2}$.
\end{lemma}
\begin{proof} 
	As $a_{\mathbb{K}}^{2}(n)$ and $(-1)^{n}\sum_{d|n}(-1)^{d}d^{3}$ are multiplicative, we have
	\begin{align*}
		F(s) &= L(s,a_{\mathbb{K}}^{2}(n)r_{8}(n))=16L(s,a_{\mathbb{K}}^{2}(n)g(n))\\
		&=16\sum_{n=1}^{\infty} \frac{a_{\mathbb{K}}^{2}(n)\big((-1)^{n}\sum_{d|n}(-1)^{d}d^{3}\big)}{n^{s}}\\
		&=16\prod_{p}\Big(1-\frac{a_{\mathbb{K}}^{2}(p)((-1)^{p+1}+p^{3})}{p^{s}}\Big)^{-1}.
	\end{align*}
	For $p\neq 2$, we let
	\begin{align*}
		A_{p}(s)&=1+\frac{a_{\mathbb{K}}^{2}(p)g(p)}{p^{s}}+\frac{a_{\mathbb{K}}^{2}(p^{2})g(p^{2})}{p^{2s}}+\cdots\\
		&=1+\frac{a_{\mathbb{K}}^{2}(p)(1+p^{3})}{p^{s}}+\frac{a_{\mathbb{K}}^{2}(p^{2})(1+p^{3}+p^{6})}{p^{2s}}+\cdots.
	\end{align*} 
	For $p=2$, we have
	\begin{align*}
		A_{2}(s)&=1+\frac{a_{\mathbb{K}}^{2}(2)(1+2^{3})}{2^{s}}+\frac{a_{\mathbb{K}}^{2}(2^{2})(1+2^{3}+2^{6})}{2^{2s}}+\cdots,
	\end{align*}
	and let
	\begin{align*}
		B_{2}(s)&=1+\frac{a_{\mathbb{K}}^{2}(2)(-1+2^{3})}{2^{s}}+\frac{a_{\mathbb{K}}^{2}(2^{2})(-1+2^{3}+2^{6})}{2^{2s}}+\cdots.
	\end{align*}
	Then,
	\begin{align*}
		F(s)&=16 B_{2}(s)\prod_{p\neq 2} A_{p}(s)\\
		&=16\frac{B_{2}(s)}{A_{2}(s)}\prod_{p} A_{p}(s) \\
		&=16\frac{B_{2}(s)}{A_{2}(s)}\prod_{p}\Big(1+\frac{a_{\mathbb{K}}^{2}(p)(1+p^{3})}{p^{s}}+\frac{a_{\mathbb{K}}^{2}(p^{2})(1+p^{3}+p^{6})}{p^{2s}}+\dots\Big)\\
		&=16\frac{B_{2}(s)}{A_{2}(s)}\prod_{p}\Big(1+\frac{(1+\lambda_{f}(p))^{2}(1+p^{3})}{p^{s}}+\frac{(1+\lambda_{f}(p^{2}))^{2}(1+p^{3}+p^{6})}{p^{2s}}+\dots\Big)\\
		&=16\frac{B_{2}(s)}{A_{2}(s)}\prod_{p}\Big(1+\frac{(2+2\lambda_{f}(p)+\lambda_{sym^{2}f}(p))(1+p^{3})}{p^{s}}+\dots\Big)\\
		&=16\frac{B_{2}(s)}{A_{2}(s)}\prod_{p}\Big(1+\frac{2+2\lambda_{f}(p)+\lambda_{sym^{2}f}(p)+2p^{3}+2p^{3}\lambda_{f}(p)+p^{3}\lambda_{sym^{2}f}(p)}{p^{s}}+\dots\Big)\\
		&=16\frac{B_{2}(s)}{A_{2}(s)}\zeta^{2}(s-3)L^{2}(s-3,f)L(s-3,sym^{2}f)\zeta^{2}(s)L^{2}(s,f)L(s,sym^{2}f)B(s),
	\end{align*}
	where $B(s)$ is a harmless factor which is absolutely convergent for $Re(s)>\frac{7}{2}$. The Dirichlet series $F(s)$ has a pole of order $2$ at $s=4$.\\
	
	\newpage
	\noindent Let
	\begin{center}
		$G(s)=16\zeta^{2}(s)L^{2}(s,f)L(s,sym^{2}f)B(s)$.
	\end{center}
	Now, we can write\\
	\begin{center}
		$F(s)=\frac{B_{2}(s)}{A_{2}(s)}\zeta^{2}(s-3)L^{2}(s-3,f)L(s-3,sym^{2}f)G(s)$.
	\end{center}
	We would like to show that $0<\Big|\frac{B_{2}(s)}{A_{2}(s)}\Big|\leq 1$ for $Re(s) =\frac{7}{2}$.\\
	As $\mathbb{K}/\mathbb{Q}$ is a field extension of degree $3$, the ideal generated by a prime number $p$ in $\mathbb{Z}$ can be expressed as a product of prime ideals in $\mathcal{O}_{\mathbb{K}}$ as
	\begin{center}
		$p\mathcal{O}_{\mathbb{K}}=\mathcal{P}_{1}^{e_{1}}\,\mathcal{P}_{2}^{e_{2}}\dots\mathcal{P}_{r}^{e_{r}}$,
	\end{center}
	with $e_{1}f_{1}+e_{2}f_{2}+\dots+e_{r}f_{r}=3$, where $e_{i}$ is the ramification index and $f_{i}$ is the degree of $\mathcal{O}_{\mathbb{K}}/\mathcal{P}_{i}$ over the field $\mathbb{Z}_{p}$ (see, Page.no. $52$, \cite{milne}). This implies that $r\leq 3$ and $p\mathcal{O}_{\mathbb{K}}=\mathcal{P}_{1}^{e_{1}}\,\mathcal{P}_{2}^{e_{2}}\,\mathcal{P}_{3}^{e_{3}}$.\\
	In particular for the ideal $(2)$ generated by $2$
	\begin{center}
		$(2)\mathcal{O}_{\mathbb{K}}=\mathcal{P}_{1}^{e_{1}}\,\mathcal{P}_{2}^{e_{2}}\,\mathcal{P}_{3}^{e_{3}}$.
	\end{center}
	We have the following three cases:\\
	\textbf{Case (i):} If $(2)\mathcal{O}_{\mathbb{K}}$ is a prime ideal in $\mathcal{O}_{\mathbb{K}}$, then $a_{\mathbb{K}}(2^{n})\leq1$.\\
	
	\noindent\textbf{Case (ii):} If $(2)\mathcal{O}_{\mathbb{K}}=\mathcal{P}_{1}\mathcal{P}_{2}\mathcal{P}_{3}$, $\mathcal{P}_{1},\,\mathcal{P}_{2},\,\mathcal{P}_{3}$ are distinct then $a_{\mathbb{K}}(2^{n})\leq 1$.\\
	
	\noindent\textbf{Case (iii):} If $(2)\mathcal{O}_{\mathbb{K}}=\mathcal{P}_{1}^{e_{1}}\mathcal{P}_{2}^{e_{2}}$, then $e_{1}+e_{2}=3$. Assume that $e_{1}=1$ and $e_{2}=2$ then $N((2)\mathcal{O}_{\mathbb{K}})=2^{3}$. Hence we have, if    $I$ is an ideal of $\mathcal{O}_{\mathbb{K}}$ such that $2^{n}=N(I)$, then  $I=\mathcal{P}_{1}^{e_{1}}\mathcal{P}_{2}^{e_{2}}$ such that $e_{1}+2e_{2}=3$. Therefore the number of integral ideals in $\mathcal{O}_{\mathbb{K}}$ whose norm is $2^{n}$, is less than or equal to $\Big[\frac{n}{2}\Big]+1\leq n$, for $n\geq 2.$\\
	
	\noindent For $s=\frac{7}{2}$, we show that $A_{2}(\frac{7}{2})$ converges absolutely. 
	\begin{align*}
		A_{2}(\frac{7}{2})&=1+\frac{a_{\mathbb{K}}^{2}(2)(2^{3}+1)}{2^{\frac{7}{2}}}+\frac{a_{\mathbb{K}}^{2}(2^{2})(2^{6}+2^{3}+1)}{2^{2.\frac{7}{2}}}+\cdots\\
		&=1+\frac{a_{\mathbb{K}}^{2}(2)(2^{6}-1)}{(2^{3}-1)2^{\frac{7}{2}}}+\frac{a_{\mathbb{K}}^{2}(2^{2})(2^{9}-1)}{(2^{3}-1)2^{2.\frac{7}{2}}}+\cdots\\
		&\leq 1+\frac{1}{7}\Big\{\frac{a_{\mathbb{K}}^{2}(2)2^{6}}{2^{\frac{7}{2}}}+\frac{a_{\mathbb{K}}^{2}(2^{2})2^{9}}{2^{2.\frac{7}{2}}}+\frac{a_{\mathbb{K}}^{2}(2^{3})2^{12}}{2^{3.\frac{7}{2}}}+\cdots\Big\}\\
		&=1+\frac{8}{7}\Big\{\frac{a_{\mathbb{K}}^{2}(2)}{2^{\frac{1}{2}}}+\frac{a_{\mathbb{K}}^{2}(2^{2})}{2}+\frac{a_{\mathbb{K}}^{2}(2^{3})}{2^{\frac{3}{2}}}+\cdots\Big\}\\
		&\leq 1+\frac{8}{7}\Big\{\frac{a_{\mathbb{K}}^{2}(2)}{\sqrt{2}}+\frac{a_{\mathbb{K}}^{2}(2^{2})}{(\sqrt{2})^{2}}+\frac{a_{\mathbb{K}}^{2}(2^{3})}{(\sqrt{2})^{3}}+\cdots\Big\}.
	\end{align*}
	We know that, $a_{\mathbb{K}}(2^{n})\leq n$, therefore $(\sqrt{2})^{n}=2^{\frac{n}{2}}=e^{\frac{n\log{2}}{2}}>\frac{(\frac{n\log{2}}{2})^{4}}{4!}$, which implies that $\frac{1}{(\sqrt{2})^{n}}<\frac{2.4!}{(n\log{2})^{4}}$.\\
	Hence,
	\begin{align*}
		\sum_{n=0}^{\infty} \frac{a_{\mathbb{K}}^{2}(2^{n})}{(\sqrt{2})^{n}}&< 1+\frac{8}{7}\sum_{n=1}^{\infty}\frac{a_{\mathbb{K}}^{2}(2^{n}).2.4!}{(n\log{2})^{4}}\\
		&\leq 1+\frac{8}{7}\sum_{n=1}^{\infty} \frac{(n)^{2}.2.4!}{n^{4}(\log{2})^{4}}\\
		&=1+\frac{2^{4}.4!}{7(\log 2)^{4}}\sum_{n=1}^{\infty} \frac{1}{n^{2}}\\
		&\leq\frac{2^{4}.4!}{7(\log 2)^{4}}\Big(1+\sum_{n=1}^{\infty}\frac{1}{n^{2}}\Big).
	\end{align*}
	Hence $1+\sum_{n=0}^{\infty} \frac{a_{\mathbb{K}}^{2}(2^{n})}{(\sqrt{2})^{n}}$ is a convergent series, and hence $A_{2}(\frac{7}{2})$ converges.\\
	As, $\frac{a_{\mathbb{K}}^{2}(2^{n})(2^{3n}+2^{3(n-1)}+\dots+2^{3}-1)}{2^{\frac{7}{2}n}}\leq\frac{a_{\mathbb{K}}^{2}(2^{n})(2^{3n}+2^{3(n-1)}+\dots+2^{3}+1)}{2^{\frac{7}{2}n}}$ for every $n\in \mathbb{N}$, and $\frac{a_{\mathbb{K}}^{2}(2^{n})(2^{3n}+2^{3(n-1)}+\dots+2^{3}-1)}{2^{\frac{7}{2}n}}>0$, and hence $B_{2}(\frac{7}{2})$ also converges to real number $>0$. Therefore by the Theorem $11.4$ \cite{apostol}, $0<|B_{2}(s)|<|A_{2}(s)|$ for $Re(s)>\frac{7}{2}$ and hence $0<\big|\frac{B_{2}(s)}{A_{2}(s)}\big|\leq 1$ for $Re(s)>\frac{7}{2}$ .
	This implies that, $\frac{B_{2}(s)}{A_{2}(s)}G(s)\neq 0$, and is absolutely convergent for $Re(s)>\frac{7}{2}$.\\
	The Dirichlet's series $F(s)$ has a pole of order $2$ at $s=4$ and converges absolutely for $Re(s)>4$. It does not vanish on the line $Re(s)=4.$
\end{proof}

In the course of the proof of the theorem, we shall make use of certain well-known estimates for the Riemann zeta function and associated $L$-functions, which we state below. 
\begin{lemma}\cite{Kram}
	\textit{	For $U\geq U_{0},$ where $U_{0}$ is sufficiently large, there exists a point $T^*\in (U,\,2U)$ such that
		$$\max\limits_{\sigma \geq \frac{1}{2}}|\zeta(\sigma\pm iT^*)|~\leq~exp(C(\log\log{U})^{2})\ll U^{\epsilon}.$$}
\end{lemma}
\begin{lemma} \cite{Ylin}
	\textit{	For any $\epsilon>0,$ we have
		\begin{align*}
			\zeta(\sigma+it)&\ll_{\epsilon} (1+|t|)^{max\{\frac{13}{42}(1-\sigma),0\}+\epsilon},
		\end{align*}
		uniformly for $\frac{1}{2}\leq \sigma \leq 1+\epsilon$ and $|t|\geq 1.$}
\end{lemma}
\begin{lemma}\cite{Aivik}
	\textit{For any $\epsilon>0$, we have 
		\begin{align*}
			L(\sigma+it,f)&\ll (|t|+10)^{max\{\frac{2}{3}(1-\sigma),0\}+\epsilon},
		\end{align*}
		uniformly for $\frac{1}{2}\leq \sigma\leq 1+\epsilon$ and $|t|\geq 10$.}
\end{lemma}
\begin{lemma}\cite{Ylin}
	\textit{For any $\epsilon>0$, we have
		\begin{align*}
			L(\sigma+it,sym^{2}f)&\ll (|t|+10)^{max\{\frac{6}{5}(1-\sigma),0\}+\epsilon},
		\end{align*}
		uniformly for $\frac{1}{2}\leq \sigma\leq 1+\epsilon$ and $|t|\geq 10$.}
\end{lemma}
\begin{lemma}\cite{AS}
	\textit{For $T\geq 100$, we have 
		\begin{align*}
			\int_{10}^{T}|L(\frac{1}{2}+\epsilon+it,sym^{2}f)|^{2}\,dt&\ll T^{\frac{3}{2}+\epsilon}.
		\end{align*}
	}
\end{lemma}

\begin{lemma}\cite{Lsankar}
	\textit{
		For $T\geq 100$, we have 
		\begin{align*}
			\int_{10}^{T}|L(\frac{1}{2}+\epsilon+it,f)|^{2}\,dt&\ll T^{1+\epsilon}.
		\end{align*}
	}
\end{lemma}
\section{Proof of the main Theorem}
To start with, we make the choice of our $T=\pm T^{\star}\in (U,2U)$ to satisfy Lemma $2.3$. By applying Perron's formula \cite{Agranville} to $F(s)$ with the choice $\eta= 4+\epsilon$ and $1\leq T \leq x$, we get
\begin{align*}
	\sum\limits_{\substack{n\leq x \\
			n=n_{1}^{2}+n_{2}^{2}+\dots+n_{8}^{2}\\(n_1,n_2,\dots,n_{8})\in \mathbb{Z}^{8}}} a_{\mathbb{K}}^{2}(n)&= \frac{1}{2\pi i} \int_{4+\epsilon-iT}^{4+\epsilon+iT} F(s) \frac{x^{s}}{s}\,ds+O\Big(\frac{x^{4+\epsilon}}{T}\Big).
\end{align*}
By moving the line of integration to $Re(s)=\frac{7}{2}+\epsilon$. We consider the rectangle $R$ with the vertices $\frac{7}{2}+\epsilon\pm iT$ and $4+\epsilon\pm iT$. In the rectangle $R$, $F(s)\frac{x^{s}}{s}$ has a pole of multiplicity $2$ at $s=4$ coming from the factor $\zeta^{2}(s-3)$. Thus, by using Cauchy's residue theorem on $R$, we get
\begin{align*}
	\frac{1}{2\pi i}\int_{4+\epsilon-iT}^{4+\epsilon+iT}F(s)\frac{x^{s}}{s}\,ds&=M(x)+\frac{1}{2\pi i}\Big\{-\int_{4+\epsilon+iT}^{\frac{7}{2}+\epsilon+iT}-\int_{\frac{7}{2}+\epsilon+iT}^{\frac{7}{2}+\epsilon-iT}\\
	&~~~~~~~~~~~~~~~~~~~~~~~-\int_{\frac{7}{2}+\epsilon-iT}^{4+\epsilon-iT}\Big\}F(s) \frac{x^{s}}{s}\,ds\\
	&=M(x)+\frac{1}{2\pi i}(J_{1}+J_{2}+J_{3})~~~~~~~(\text{say}),
\end{align*}
here $M(x)$ is the main term i.e, the residue of $F(s)\frac{x^{s}}{s}$ at $s=4$, and is given by $Cx^{4}P_{1}(\log{x})$ where $P_{1}(\log{x})$ is a polynomial of degree $1$ in $\log{x}$.\\

Let $C_{H}$ be the horizontal contribution of $F(s)\frac{x^{s}}{s}$ along the lines $\frac{7}{2}+\epsilon-iT$ to $4+\epsilon-iT$ and $\frac{7}{2}+\epsilon+iT$ to $4+\epsilon+iT$.
\begin{align*}
	|C_{H}|& \ll \Bigg|\int_{\frac{7}{2}+\epsilon+iT}^{4+\epsilon+iT}\,F(s)\frac{x^{s}}{s}\,ds\Bigg|\\
	&\ll \int_{\frac{7}{2}+\epsilon+iT}^{4+\epsilon+iT} |\zeta(s-3)|^{2}|L(s-3,f)|^{2}|L(s-3,sym^{2}f)|\Big|\frac{A_{2}(s)}{B_{2}(s)}\Big||G(s)|\frac{|x^{s}|}{|s|}\,ds.
\end{align*}
As $\Big|\frac{A_{2}(s)}{B_{2}(s)}\Big||G(s)|$ is bounded for $Re(s)>\frac{7}{2}$, we have
\begin{align*}
	|C_{H}|&\ll\frac{x^{3}}{T}\int_{\frac{1}{2}+\epsilon}^{1+\epsilon}|\zeta(\sigma+iT)|^{2}|L(\sigma+iT,f)|^{2}|L(\sigma+iT,sym^{2}f)|\,x^{\sigma}\,d\sigma.
\end{align*}
By using Lemmas $2.3$, $2.5$ and $2.6$ we get
\begin{align*}
	|C_{H}|& \ll \frac{x^{3}}{T}\,T^{2\epsilon}\int_{\frac{1}{2}+\epsilon}^{1+\epsilon} T^{\frac{4}{3}(1-\sigma)+\epsilon}\,T^{\frac{6}{5}(1-\sigma)+\epsilon}\,x^{\sigma}\,d\sigma\\
	&\ll x^{3}T^{\frac{23}{15}+4\epsilon}\int_{\frac{1}{2}+\epsilon}^{1+\epsilon} \Big(\frac{x}{T^{\frac{38}{15}}}\Big)^{\sigma}\, d\sigma.
\end{align*}
Since $\Big(\frac{x}{T^{\frac{38}{15}}}\Big)^{\sigma}$ is a monotonic function of $\sigma$ on $[\frac{1}{2}+\epsilon,1+\epsilon]$, it follows that the maximum is attained at the endpoints of the interval $[\frac{1}{2}+\epsilon,1+\epsilon]$. Thus,
\begin{align*}
	|C_{H}|&\ll x^{3}T^{\frac{23}{15}+4\epsilon} \Big(\Big(\frac{x}{T^{\frac{38}{15}}}\Big)^{1+\epsilon}+\Big(\frac{x}{T^{\frac{38}{15}}}\Big)^{\frac{1}{2}+\epsilon}\Big)\\
	& \ll \frac{x^{4+\epsilon}}{T}+x^{\frac{7}{2}+\epsilon}T^{\frac{4}{15}+\epsilon}.
\end{align*}
Now we estimate the contribution along the left vertical line $(J_{2})$ in absolute value
\begin{align*}
	|C_{V}|&=\Big|\int_{\frac{7}{2}+\epsilon-iT}^{\frac{7}{2}+\epsilon+iT} F(s)\frac{x^{s}}{s}\,ds\Big|\\
	&=\int_{\frac{7}{2}+\epsilon-iT}^{\frac{7}{2}+\epsilon+iT} |\zeta(s-3)|^{2}|L(s-3,f)|^{2}|L(s-3,sym^{2}f)|\Big|\frac{A_{2}(s)}{B_{2}(s)}\Big||G(s)|\frac{|x^{s}|}{|s|}\,ds\\
	&\ll x^{\frac{7}{2}+\epsilon}+x^{\frac{7}{2}+\epsilon}\int_{10}^{T} |\zeta(\frac{1}{2}+\epsilon+it)|^{2}|L(\frac{1}{2}+\epsilon+it,f)|^{2}|L(\frac{1}{2}+\epsilon+it,sym^{2}f)|\frac{1}{t}\,dt\\
	&\ll x^{\frac{7}{2}+\epsilon}+x^{\frac{7}{2}+\epsilon} \max_{10\leq t\leq T}\{|\zeta(\frac{1}{2}+\epsilon+it)|^{2}|L(\frac{1}{2}+\epsilon+it,f)|\}\\
	&~~~~~~~~~~\times \int_{10}^{T} |L(\frac{1}{2}+\epsilon+it,f)|\,|L(\frac{1}{2}+\epsilon+it,sym^{2}f)|\frac{1}{t}\,dt.
\end{align*}
By using Lemmas $2.4$, $2.5$, $2.7$ and $2.8$ we get
\begin{align*}
	|C_{V}| &\ll x^{\frac{7}{2}+\epsilon}+ x^{\frac{7}{2}+\epsilon}\,(T^{\frac{13}{84}})^{2}\,T^{\frac{1}{3}+\epsilon}\Big(\int_{10}^{T}\frac{|L(\frac{1}{2}+\epsilon+it,f)|^{2}}{t}\,dt\Big)^{\frac{1}{2}}\Big(\int_{10}^{T}\frac{|L(\frac{1}{2}+\epsilon+it,sym^{2}f)|^{2}}{t}\,dt\Big)^{\frac{1}{2}}\\
	& \ll x^{\frac{7}{2}+\epsilon}+ x^{\frac{7}{2}+\epsilon}\,T^{\frac{9}{14}+\epsilon} T^{\frac{1}{4}+\epsilon} \\
	&\ll x^{\frac{7}{2}+\epsilon}\,T^{\frac{25}{28}+\epsilon}.
\end{align*}
Therefore, the total error is (note that $\frac{4}{15}<\frac{25}{28}$)
\begin{align*}
	|C_{H}|+|C_{V}|+O\Big(\frac{x^{4+\epsilon}}{T}\Big)&\ll\frac{x^{4+\epsilon}}{T}+x^{\frac{7}{2}+\epsilon}T^{\frac{25}{28}+\epsilon}+O\Big(\frac{x^{4+\epsilon}}{T}\Big).
\end{align*}
To get an optimal estimate, we choose $T=x^{\frac{14}{53}}$ so that we obtain
\begin{align*}
	|C_{H}|+|C_{V}|+O\Big(\frac{x^{4+\epsilon}}{T}\Big)&= O(x^{\frac{198}{53}+\epsilon}).
\end{align*}
Therefore, 
\begin{align*}
	\sum\limits_{\substack{n=\sum_{i=1}^{8}n_{i}^{2} \leq x\\(n_1,n_2,n_{3},n_{4},n_{5},n_{6},n_{7},n_{8})\in \mathbb{Z}^{8}}} a_{\mathbb{K}}^{2}(n)&= Cx^{4}P_{1}(\log{x})+O(x^{\frac{198}{53}+\epsilon}).
\end{align*}
This proves the main Theorem.
\hfill $\Box$

\end{document}